\documentclass[a4paper]{amsart}
\usepackage{amssymb}
\usepackage[hyphens]{url} 
\usepackage{color}
\usepackage[english]{babel}
\newcommand{\m}{\mathcal{M}}
\newcommand{\mo}{\mathcal{M}_0}
\newcommand{\mt}{\mathcal{M}_t}
\newtheorem{theorem}{Theorem}[section]
\newtheorem{lem}[theorem]{Lemma}
\newtheorem{prop}[theorem]{Proposition}
\newtheorem{cor}[theorem]{Corollary}
\newcommand{\xb}{\bar{x}}

\newcommand{\ls}{\Delta_{\dot{\sigma}}}
\newcommand{\ts}{tr_{\dot{\sigma}}}

\newcommand{\qb}{\bar{q}}

\newcommand{\ld}{\lambda}

\theoremstyle{definition}

\theoremstyle{remark}
\newtheorem{remark}[theorem]{Remark}

\numberwithin{equation}{section}

\begin{document}
\title[Volume preserving  flow by powers of symmetric polynomials]{Volume preserving flow by powers of symmetric polynomials in the principal curvatures }
\author[Maria Chiara Bertini]{Maria Chiara Bertini\textsuperscript{1}} \thanks{\textsuperscript{1} Dipartimento di Matematica e Fisica, Universit\`a di Roma ``Roma Tre'', Largo San Leonardo Murialdo 1, 00146, Roma, Italy. E-mail: bertini@mat.uniroma3.it} 
 \author[Carlo Sinestrari]{Carlo Sinestrari\textsuperscript{2}}  \thanks{\textsuperscript{2} (correspond. author) Dipartimento di Ingegneria Civile e Ingegneria Informatica, Universit\`a di Roma ``Tor Vergata'', Via Politecnico 1, 00133, Roma, Italy. E-mail: sinestra@mat.uniroma2.it}
\begin{abstract}
We study a volume preserving curvature flow of convex hypersurfaces, driven by a power of the $k$-th elementary symmetric polynomial in the principal curvatures. Unlike most of the previous works on related problems, we do not require assumptions on the curvature pinching of the initial datum. We prove that the solution exists for all times and that the speed remains bounded and converges to a constant in an integral norm. In the case of the volume preserving scalar curvature flow, we can prove that the hypersurfaces converge smoothly to a round sphere.
\end{abstract}
\maketitle
 \bigskip
\section{Introduction}
Let $\m$ be an oriented, compact  $n$-dimensional  manifold without boundary. We embed $\m$ in the Euclidean $(n+1)$-space by $F_0:\m\rightarrow \mathbb{R}^{n+1}$, and denote its image by $\mo=F_0(\m)$. We assume that $\mo$ is strictly convex.
Then we consider a family of maps $F:\m\times[0,T)\rightarrow\mathbb{R}^{n+1}$, with $F_t:=F(\cdot,t):\m\rightarrow \mathbb{R}^{n+1}$  satisfying \smallskip
\begin{equation}\label{fl}
\left\{
\begin{array}{l}
\partial_t F(x,t)=[-\sigma(x,t)+h(t)]\nu(x,t) \medskip\\
F(x,0)=F_0(x),\\
\end{array}
\right.
\end{equation}
where:
\begin{itemize}
\item  $\nu$ is  the outer unit normal to the evolving hypersurface $\mt:=F_t(\m)$;
\item $\sigma(x,t)=E^{\alpha}_k$ with $\alpha >0$ and $E_k$ is the $k$-th symmetric polynomial in the principal curvatures, i.e.
$$E_k(x,t)=\sum_{1\leq i_1<\dots<i_k\leq n }\lambda_{i_1}(x,t)\dots\lambda_{i_k}(x,t),$$ with $\lambda_1,\dots,\lambda_n$ the principal curvatures of $\mt$ and  $k \in \{1,\dots,n \}$;
\item The function $h(t)$ is defined as
 \begin{equation}\label{vpr}h(t):=\frac{1}{A(\mt)}\int_{\mt} \sigma d\mu,\end{equation}
 where $A(\mt)$ is the $n$-dimensional measure of $\mt$.
\end{itemize}
Such a definition of $h(t)$ ensures that the  volume $Vol(\Omega_t)$ is preserved by the flow, where $\Omega_t$ is the $(n+1)$-dimensional region bounded by $\mt$.

We will prove the following result.
\begin{theorem}\label{mt}Let  $F_0:\m\rightarrow \mathbb{R}^{n+1}$, with $n \geq 1$, be a smooth embedding of an oriented, compact   $n$-dimensional manifold without boundary, such that $F_0(\m)$ is strictly convex. Then 
\begin{itemize}
\item [(i)] the flow \eqref{fl} has a unique smooth solution, which remains strictly convex and exists for all times $t\in[0,\infty)$;
\item[(ii)] the speed $\sigma$ is bounded from above along the flow and converges to its mean value in $L^2$-norm
$$
\int_{\mt} |\sigma - h(t)|^2 \, d\mu \to 0 \mbox{ as }t \to \infty;
$$
\item[(iii)]  if $\alpha=1$ and $k=2$, i.e. $\sigma$ is the scalar curvature, then $\mt$ converges smoothly with exponential rate to a round sphere which encloses the same volume as $\m_0$.
\end{itemize}
\end{theorem}

Some cases of these flows have been studied in the past literature. When $\alpha=1/k$, they belong to the class considered by McCoy \cite{Mc2} who proved smooth convergence to a sphere. The same result was obtained by Cabezas-Rivas and the second author \cite{CaSi} for $\alpha>1/k$ assuming a pinching condition on the principal curvatures of the initial value. More recently \cite{Si}, the second author has considered the case $k=1$ for any $\alpha>0$, proving convergence to a sphere for general strictly convex data.

In general, there is a wide literature about curvature flows of convex hypersurfaces, both in the standard version (without the forcing term) and in the constrained one. The starting point is the result by Huisken \cite{Hu1} who proved that any closed convex hypersurface moving by mean curvature flow shrinks to a point in finite time with a spherical profile. The corresponding result in the volume preserving version is also due to Huisken, who in \cite{Hu2} showed that, starting from a closed  convex datum, the solution exists for all times and converges smoothly to a round sphere. Since then, many authors have studied curvature flows where the speed is a symmetric homogeneous function of the principal curvatures. The case of homogeneity degree equal to one is better known and investigated. There are results of convergence to a round point (for standard flows) or to a round sphere (for constrained flows) for very general speed functions, see for example \cite{AnMcZh,Mc2}. When the homogeneity degree is greater than one, the analysis is more difficult and the corresponding results have usually been proved only under some additional hypotheses. Typically, authors have imposed suitable pinching conditions on the principal curvatures of the initial data, or have restricted the analysis to particular dimensions or to data with symmetries, see for example \cite{AlSi,AM,Ch,Schn,Sch2} for standard flows, or \cite{CaSi,Mc3} for volume preserving flows. In fact, in most cases convergence results are obtained by considering the invariance or improvement of the curvature pinching, a property which is difficult to study or may even fail for general convex hypersurfaces evolving by speeds with general homogeneity.

Recently, in the above recalled paper \cite{Si}, it was observed that certain flows with homogeneity different from one enjoy better properties in the constrained case than in the standard one. In fact, a typical property of constrained flows is the monotonicity of a suitable isoperimetric ratio, and this provides an alternative technique to the study of the behaviour of convex solutions. In \cite{Si}, the property was used to control the outer and inner radii and to find an upper curvature bound, with no need to assume curvature pinching. A similar procedure had been used previously in the study of the anisotropic mean curvature flow \cite{An2}. Later, in \cite{BeSi}, the volume preserving flow driven by general nonhomogeneous functions of the mean curvature was studied, finding again convergence to a round sphere. An important step in this last paper was the derivation of a lower bound on the speed which avoids the use of regularity results from degenerate parabolic equations as it was done in \cite{CaSi,Sch2,Si}. 

While \cite{BeSi,Si} studied speeds given by functions of the mean curvature, we consider here general symmetric polynomials. The flows we study are related to the mixed volumes, which are quantities that generalize the notion of area and volume of a convex body, and that can be expressed as boundary integrals of the polynomials $E_k$. Using the monotonicity of a suitable mixed volume under the flow, we obtain a bound on the inner and outer radius of our hypersurface, which in turn implies a uniform upper bound on the speed and the global existence of the solution. In contrast with the cases considered in \cite{BeSi,Si}, the bound on the speed no longer implies a bound on the curvatures and the study of the asymptotic behaviour is more difficult. We can show that the speed tends to a constant in the $L^2$-norm along suitable time sequences; since the only closed hypersurfaces with  $E_k$ constant are the spheres, this suggests that the flow should converge to a sphere. However, in the general case we lack a bound from below on the curvature and our estimates are not strong enough to guarantee the existence of a smooth limit as time goes to infinity. We can solve these difficulties in the case of the scalar curvature flow, where the evolution equation of the mean curvature has a favorable structure, and we can prove a uniform bound on the principal curvatures using the boundedness of the speed and a maximum principle argument. We then prove convergence to a sphere by combining an argument of Ros \cite{Ro} for constant mean curvature hypersurfaces with a recent $L^1$-stability estimate due to Magnanini and Poggesi \cite{MaPo}.
\bigskip

\noindent{\bf Note}. The results of this paper are part of the first author's PhD thesis. Shortly before submitting the paper, a related preprint by B. Andrews and Y. Wei \cite{AW} has appeared, which uses different methods and obtains further convergence results for the flows considered here.

\section{Preliminaries}\label{Preliminaries}
\subsection*{Notations}
Let $F:\m\rightarrow \mathbb{R}^{n+1}$ be an embedded hypersurface with local coordinates $(x^1,\cdots, x^n)$. We always assume $n>1$. We endow $\m$ with the induced metric $g=(g_{ij})$ given by
$$g_{ij}=\left(\frac{\partial F}{\partial x^i},\frac{\partial F}{\partial x^j}\right)$$ where $(\cdot,\cdot)$ is the standard Euclidean inner product. The inverse of $g_{ij}$ will be written as $g^{-1}=(g^{ij})$. We also denote respectively by $\nabla$ and $A=(h_{ij})$ the Levi-Civita connection and the second fundamental form of $\m$, while the measure is $d\mu=\sqrt{\det g_{ij}}\, dx$. The principal curvatures are denoted by $\lambda_1,\dots,\lambda_n$, and the mean curvature by $H=\lambda_1+\dots+\lambda_n$. We say that the hypersurface is strictly convex if all $\lambda_i$'s are positive. 

As usual, we always sum on repeated indices, and we lower or lift
tensor indices via $g$, e.g. the Weingarten operator is given by $$h^i_j=h_{kj}g^{ik}.$$
Given tensors $T=(T^{i_1\dots i_s}_{j_1\dots j_r})$ and $S=(S^{i_1\dots i_s}_{j_1\dots j_r})$ on $\m$, we use brackets to denote their inner product
$$\langle T,S\rangle=T^{i_1\dots i_s}_{j_1\dots j_r}S_{i_1\dots i_s}^{j_1\dots j_r}.$$
In particular, the square of the norm is given by
$$|T|^2=T^{i_1\dots i_s}_{j_1\dots j_r}T_{i_1\dots i_s}^{j_1\dots j_r}.$$
%In the following we will use the same notation for the norm with respect to $g$ and the absolute value of a scalar quantity.
Given a point $q\in \mathbb{R}^{n+1}$, the {\em support function} of $\m$ with respect to $\xb$ is
$$u_{q}(x):=(F(x)-q,\nu(x)),$$
where $\nu(x)$ is the outer unit normal vector of $\m$ at the point $x$. The subscript $q$ will be omitted whenever there will be no ambiguity. 

%For a given $k=1,\dots,n$, the $k$-th elementary symmetric polynomial of the principal curvatures of $\m$ is
%$$E_k(\lambda_1,\dots,\lambda_n)=\sum_{1\leq i_1< \dots< i_k \leq n}\lambda_{i_1} \cdots \lambda_{i_k}.$$
It is convenient to define the symmetric polynomials also for $k=0,n+1$ setting $E_0 \equiv 1$ and $E_{n+1} \equiv 0$.
In order to simplify some formulas, we introduce the normalized symmetric polynomials
$$
\tilde E_k := \binom{n}{k}^{-1} E_k, \qquad k=0,\dots,n,
$$
which satisfy  $\tilde E_k(1,\dots,1)=1$. For the purposes of this paper, these functions will only be evaluated in the positive cone $\Gamma_+:= \{ (\ld_1,\dots,\ld_n) \ : \ \ld_1 > 0, \dots \ld_n>0 \}.$  

The polynomials $E_k$ and $\tilde E_k$ can be also regarded as a function of the Weingarten operator of $\m$. We will use the same symbol in the two cases, since the meaning will be clear from the context.  We recall some well known properties, see e.g. \S XV.4 in \cite{Li} or Lemma 2.1 in \cite{CaSi}.
\begin{lem}\label{dispolsim} The following relations hold, for any $k=1,\dots,n$ and $(\ld_1,\dots,\ld_n) \in \Gamma_+$.
\begin{trivlist}
\item $(i)$ $\frac{\partial E_k}{\partial \lambda_i}\lambda_i^2=HE_k-(k+1)E_{k+1}\geq\frac{k}{n}HE_k$.
\item $(ii)$ $\tilde E_{k+1}^{\frac{1}{k+1}} \leq \tilde E_k^{\frac{1}{k}}$, with equality if and only if $\ld_1=\dots=\ld_n$ and $k<n$.
\item $(iii)$ As a function on $\mt$, $\nabla^i \frac{\partial E_k}{\partial h_{j}^i}=0$ for any $j=1,\dots,n$.
\item $(iv)$ If $\sigma=E_k^{\alpha}$, then  $\frac{\partial\sigma}{\partial\lambda_i}\lambda_i=\alpha k\sigma$. 
\end{trivlist}
\end{lem}

%Given a $n$-dimensional domain $E$, we indicate by $A(E)$ its $n$- dimensional measure, while given a $(n+1)$-dimensional domain $F$ we indicate by $Vol(F)$ its $(n+1)$-dimensional measure.
%\textbf{costanti che dipendono da n}

\subsection*{Short time existence and evolution equations}
%We will consider the flow \eqref{fl} with speed $\sigma=E_k^{\alpha}$.
Given $\sigma$ as in \eqref{fl}, we shall denote $\ls=\dot{\sigma}^{ij}\nabla_i\nabla_j$, where $\dot{\sigma}^{ij}=\frac{\partial \sigma}{\partial h_{ij}}$. Given matrices $B$ and $\tilde{B}$, we define
$$tr_{\dot{\sigma}}(B)=\dot{\sigma}^{ij}B_{ij}\hspace{0.5cm}\text{and}\hspace{0.5cm}
\ddot{\sigma}(B,\tilde{B})=\frac{\partial^2\sigma}{\partial h_{ij}\partial h_{rs}}B_{ij}\tilde{B}_{rs}.$$
It is well known that a flow of the form \eqref{fl} without the volume preserving term is parabolic if at any point
\begin{equation}\label{par}
\frac{\partial \sigma}{\partial\lambda_i}>0,\hspace{1cm}i=1,\dots,n.
\end{equation}
In the case $\sigma=E_k^{\alpha}$, this is satisfied on any strictly convex hypersurface.
Parabolicity ensures the local existence and uniqueness of the solution.
The additional term $h(t)$ only depends on time and does not interfere with the parabolicity of the equation. Hence, we have the following result, see \cite{Hu2, HuPo,Mc2, P} for more details.
\begin{theorem} Let $F_0:\m\rightarrow \mathbb{R}^{n+1}$ be a smooth embedding of an oriented, compact  $n$-dimensional  manifold without boundary, such that $F_0(\m)$ is strictly convex. Then the flow \eqref{fl} has a unique smooth solution $\mt$ defined on a maximal time interval $[0,T)$. If $T<+\infty$, then either $\liminf_{t \to T} \min_{\mt}\frac{\partial \sigma}{\partial\lambda_i}=0$ for some $i$, or $\limsup_{t \to T} \max_{\mt} |A|^2=+\infty$.
\end{theorem}

In the next proposition we list the evolution equations for the main geometric quantities associated with the flow \eqref{fl}, which can be computed similarly to \cite{Hu1}. A detailed computation can be found in \cite{Ca}.
\begin{prop}\label{eveq}
Consider a solution of the flow \eqref{fl}, with $\sigma$ a symmetric $(\alpha k)$-homogeneous function of $\lambda_i$ and $h(t)$ a smooth function. Then the following equations hold
\begin{align*}
&\partial_t g_{ij} = 2(-\sigma +h)h_{ij},\\
&\partial_t g^{ij} = -2(-\sigma +h)h^{ij},\\
&\partial_t \nu=\nabla \sigma,\\
&\partial_t d\mu = H(-\sigma+h)d\mu,\\
%&\partial_t h_{ij} = \ls h_{ij}+\ddot{\sigma}(\nabla_i A,\nabla_j A)+\ts(h_{ml}h^l_r)h_{ij}+(h-(\alpha k+1)\sigma)h_{is}h^s _j,\\
&\partial_t h_j^i = \nabla^i\nabla_j \sigma - (h-\sigma) h^i_{l}h^l _j \\
&\hphantom{\partial_t h_j^i}= \ls h^i_j+\ddot{\sigma}(\nabla^i A,\nabla_j A)+\ts(h_{ml}h^l_r)h^i_{j}-(h+(\alpha k-1)\sigma)h^i_{s}h^s _j,\\
&\partial_t H = \ls H +tr_{g^{-1}}\left[\ddot{\sigma}(\nabla_i A,\nabla_j A)\right]+H\ts(h_{ml}h^l_r)-(h+(\alpha k-1)\sigma)|A|^2,\\
&\partial_t \sigma = \ls \sigma+(\sigma-h)\ts(h_{ml}h^l_r),\\
&\partial_t u = \ls u+\ts(h_{ml}h^l_r)u-(\alpha k+1)\sigma+h.
\end{align*}
In addition, if $h(t)$ is defined as in \eqref{vpr}, the volume of the region $\Omega_t$ enclosed by $\m_t$ is constant in time.
\end{prop}

\subsection*{Mixed volumes}
Mixed volumes are a classical notion in convex analysis, see e.g. \cite{BuZa,Sc}. We recall here the definitions and properties required for our analysis. 

Given a compact convex set $\Omega \subset \mathbb{R}^{n+1}$ and $t>0$, consider the set
$$
\Omega+tB := \{ x+ty \ : \ x \in \Omega, |y| \leq 1 \}.
$$
It can be proved, see \cite[\S 19.3.6]{BuZa} that the volume of this set is a polynomial of degree $n+1$ in $t$ and can be therefore written as
$$
{\rm Vol} (\Omega+tB) = \sum_{i=0}^{n+1} \binom{n+1}{i} \alpha_i t^i,
$$
for suitable coefficients $\alpha_i$ depending on $\Omega$. We then define the $k$-th mixed volume of $\Omega$ as $V_i(\Omega)= \alpha_{n+1-i}$, for $i=0,\dots,n+1$. It can be proved that, for any $\Omega$,
$$
V_{n+1}(\Omega)={\rm Vol}\,(\Omega), \qquad
V_n(\Omega)= A (\partial \Omega), \qquad V_0=\kappa_{n+1},
$$
where $\kappa_{n+1}$ is the volume of the unit ball in $\mathbb{R}^{n+1}$. Thus, mixed volumes can be regarded as a generalization of volume and area. They are known also as {\em cross sectional measures} or {\em quermassintegrals}.

Mixed volumes depend continuously on the set: if $\{\Omega_l\}$ is a sequence of convex sets converging to $\Omega$ in the Hausdorff topology, then
$$
V_i(\Omega_l) \to V_i(\Omega), \qquad i=1,\dots,n+1.
$$

If the convex set $\Omega$ has a smooth boundary, mixed volumes admit an equivalent characterization as boundary integrals of the elementary symmetric functions of the curvatures. In fact, it can be proved that
\begin{equation*}
V_{n-k}(\Omega)=
\begin{cases}
Vol(\Omega)\hspace{3.2cm}\text{if }k=-1\\
(n+1)^{-1}\int_{\mt}\tilde E_kd\mu \hspace{0.5cm}\text{if }k=0,1,\ldots,n-1.
\end{cases}
\end{equation*}

An important result related to the mixed volumes are the so-called Minkowski identities, which say the following. On any closed convex hypersurface $\m$ and for any $l=0,\dots,n-1$, we have
\begin{equation}\label{minkid}
\int_{\mathcal{M}} \tilde E_l d\mu= \int_{\mathcal{M}} u \, \tilde E_{l+1} \, d\mu,
\end{equation}
where $u=(F-p_0,\nu)$ is the support function centered at any fixed point $p_0$. These properties were originally proved by Minkowski and Kubota. It was later proved by Hsiung \cite{Hs} that they also hold without the convexity assumption.

Finally, we recall the following generalized isoperimetric inequality, which can be obtained as a consequence of Alexandrov-Fenchel inequality,  see e.g. \cite[\S 20]{BuZa} or \cite[\S 7]{Sc}. We have, for any compact convex set $\Omega \subset \mathbb{R}^{n+1}$, and any $1 \leq m < l \leq n+1$,
\begin{equation}\label{af}
\frac{V_l^m(\Omega)}{V_m^l(\Omega) }   \leq \frac{V_l^m(B)}{V_m^l(B) } = (\kappa_{n+1})^{m-l} ,
\end{equation}
where $B$ is the unit ball and $\kappa_{n+1}=Vol(B)$, as defined above. In addition, the inequality is strict unless $\Omega$ is a sphere, see formula (7.67) in \cite{Sc} and the following remarks.

\section{Long time existence}

\subsection*{Preservation of convexity}
\begin{prop}\label{conv}
Let $\mt$ be a convex solution of \eqref{fl} on a time interval $[0,T_0)$ and suppose that $h(t) \leq h^*$ for every $t \in [0,T_0)$ for a suitable $h^*>0$. If we set
$\ld_{\min}(t)=\min_{x \in \m_t}\ld_1(x,t)$, then we have
$$
\ld_{\min}(t) \geq \frac{1}{\ld_{\min}(0)^{-1}+{h^*}t}.
$$
\end{prop}
\begin{proof}
We follow the procedure of \cite{AnMcZh}, where the authors consider flows driven by general homogeneous speeds in the standard 
non volume-preserving case. We will recall the main steps of the proof given there and focus on the differences due to the additional term $h(t)$.

We use the Gauss map parametrization for $\mt$, given by
$$X:\mathbb{S}^n\longrightarrow\mt\subset\mathbb{R}^{n+1}$$
$$\hspace{9mm}z\hspace{1mm}\longmapsto u(z)z+\overline{\nabla}u(z),$$
which takes $z$ to the unique point in $\mt$ with outward normal direction $z$. Here $u$ is the support function $u(z)=\sup_{q\in\mt}(q,z)=(X(z),z)$, and $\overline{\nabla}$ is the gradient on the sphere $\mathbb{S}^n$ with respect to the standard metric $\bar{g}_{ij}$. If we set
$$\tau_{ij}=\overline{\nabla}_i\overline{\nabla}_j u+\bar{g}_{ij}u,$$
then it can be checked that the eigenvalues of $\tau_{ij}$ with respect to $\bar{g}$ are the principal radii of curvature $r_1,\dots,r_n$, with $r_i=\lambda_i^{-1}$.
%Here the support function is considered as a function on $\mathbb{S}^n$ instead of a function on $\m$ as elsewhere in the paper.

To describe the flow in this setting, it is convenient to define
$$\Phi(r_1,\ldots,r_n)=\left(\sigma\left(\frac 1{ r_1},\ldots,\frac 1{r_n}\right)\right)^{-1/\alpha k}.
$$
It is well known that $\Phi$ is a concave function, see for example \cite{Li}, and this property plays an important role in the study of the flow.
%= \frac{E_n(r_1,\dots,r_n)}{E_{n-k}(r_1,\dots,r_n)}
%,
%$$
%where $r_i=\lambda_i^{-1}$ are the {\em principal radii of curvature}. %
%Then $\Phi$ satisfies, see e.g. \cite{An3}
%\begin{equation}\label{propPhi}
%\frac {\partial \Phi}{\partial r_i} >0, \qquad \mbox{$\Phi$ is a concave function}. 
%\end{equation}
%
%
%For any $t \geq 0$ such that $\mt$ is convex, we can use the Gauss map parametrization given by
%$$X:\mathbb{S}^n\longrightarrow\mo\subset\mathbb{R}^{n+1}$$
%$$\hspace{9mm}z\hspace{1mm}\longmapsto u(z)z+\bar{\nabla}u(z)$$
%which takes $z$ in the unique point in $\mt$ with outward normal direction $z$. Here $u$ is the support function $u(z)=\sup_{q\in\mt}(q,z)=(X(z),z)$, and $\bar{\nabla}$ is the gradient on the sphere $\mathbb{S}^n$ with respect to the standard metric $\bar{g}_{ij}$.
%Here the support function is considered as a function on $\mathbb{S}^n$ instead of a function on $\m$ as elsewhere in the paper.
% Suppose $\mt$ convex on a certain interval $[0,T^*)$. On such interval we can adopt the Gauss map parametrization, and the flow equation is written in terms of the evolution of the support function:
%$$\partial_t u(z,t)=-\left(\sigma^*(\tau_{ij}(z,t)\right)^{-1}+h(t)$$
%where
We can also regard $\Phi$ as functions of $\tau_{ij}$ and we can write the derivatives of $\Phi$ with respect to $\tau_{ij}$ as
$$\dot{\Phi}^{lm}=\frac{\partial\Phi}{\partial\tau_{lm}}\hspace{2cm}
\ddot{\Phi}^{lm,pq}=\frac{\partial^2\Phi}{\partial\tau_{lm}\partial\tau_{pq}}.$$
Then $\tau_{ij}$ satisfies the following equation, which can be computed as in \cite[Lemma 10]{AnMcZh}.
%\begin{lem}\label{tau}
\begin{align} 
\partial_t\tau_{ij}&=\alpha k\Phi^{-\alpha k-1} [ \dot{\Phi}^{lm}\overline{\nabla}_l\overline{\nabla}_m\tau_{ij}+
\ddot{\Phi}^{lm,pq}\overline{\nabla}_i\tau_{pq}\overline{\nabla}_j\tau_{lm}
-(\alpha k+1)  \Phi^{-1} \overline{\nabla}_i\Phi \overline{\nabla}_j\Phi] \nonumber
\\
& \hspace{0.4cm}
-\alpha k\Phi^{-\alpha k-1} \dot{\Phi}^{lm}\bar{g}_{lm}\tau_{ij}
+(\alpha k-1)\Phi^{-\alpha k}\bar{g}_{ij}+h(t)\bar{g}_{ij}. \label{eqtau}
\end{align}
%\end{lem}
%\begin{proof}
%Follows from Lemma $10$ in \cite{AnMcZh}, also noticing that the additional term involving $h(t)$ comes out from
%\begin{align*}
%\partial_t\tau_{ij}&=\partial_t(\bar{\nabla}_i\bar{\nabla}_j u+\bar{g}_{ij}u)=
%\bar{\nabla}_i\bar{\nabla}_j\partial_t u+\bar{g}_{ij}\partial_t u\\
%&=-\bar{\nabla}_i\bar{\nabla}_j\left(E_{2*}\right)^{-1}
%-\bar{g}_{ij}\left(E_{2*}\right)^{-1}+h(t)\bar{g}_{ij}.
%\end{align*}
%\end{proof}
This is a parabolic equation where the first order terms give a negative contribution, due to the concavity of $\Phi$.
The sum of the first two terms in the second line is also negative definite, as shown in the proof of \cite[Theorem 5]{AnMcZh}.
In contrast to the standard case, we have here an additional positive term $h(t)\bar{g}_{ij}$. Therefore, the radii can increase, but only by an amount which is bounded as long as $h(t)$ is bounded. More precisely, if $r_0$ denotes the largest radius at time $0$, the maximum principle for tensors implies that the matrix $\tau_{ij} - (r_0 +h^* t) \bar{g}_{ij}$ remains negative definite for all times, that is, the principal radii on $\mt$ are bounded from above by $r_0 +h^* t$. The assertion follows.
\end{proof}
\begin{cor}\label{maxtime}
Let $[0,T)$ be the maximal interval of existence of the solution of \eqref{fl}. Then $\mt$ is convex for all $t \in [0,T)$. In addition, if $T<+\infty$, then the curvature of $\mt$ becomes unbounded as $t \to T$.
\end{cor}
\begin{proof}
As $h(t)$ is bounded on any compact subinterval of $[0,T)$, the convexity of $\mt$ follows from the previous proposition.
If $T<+\infty$ and the curvature is bounded, then we also have a bound on $h(t)$ for $t \in [0,T)$, and the previous proposition shows that $\mt$ remains uniformly convex as $t \to T$. This shows that the flow is uniformly parabolic and has bounded curvature on $[0,T)$. Well known regularity results, see e.g. \cite{Mc2, CaSi}, give uniform bounds on all derivatives of the solution and imply that $\mt$ converges to a smooth strictly convex limit as $t \to T$. Then we can restart the flow, in contradiction with the maximality of $T$.
\end{proof}

\subsection*{A monotone quantity} 
An important feature of the flow \eqref{fl} is the monotonicity of a suitable mixed volume, as shown by the next Lemma. 

\begin{lem}\label{dermixv} Along the flow \eqref{fl}, with $\sigma=E_k^{\alpha}$ for a given $k=1,2,\dots n$, we have
$$\frac{d}{dt} \, \int_{\mt}E_{k-1} d\mu\leq0,$$
and the inequality is strict unless $\m_t$ is a round sphere.
\end{lem}
\begin{proof}
By Proposition \ref{eveq} and Lemma \ref{dispolsim} and integrating by parts, we have
\begin{align*}
\frac{d}{dt} \, \int_{\mt}E_{k-1} d\mu
%&=\int_{\mt}E_{k}H(-E_2+h)d\mu+\int_{\mt}\frac{\partial E_k}{\partial h^i_j}\partial_t h^i_jd\mu\\
&=
\int_{\mt}\frac{\partial E_{k-1}}{\partial h^i_j}\left(\nabla^i\nabla_jE_k+(\sigma-h)h^i_mh^m_j\right)d\mu \\
&\hphantom{=}+\int_{\mt}E_{k-1}H(-\sigma+h)d\mu \\
&= \int_{\mt}\left\{ (\sigma-h)\left(HE_{k-1}-k E_{k}\right)  +E_{k-1}H(-\sigma+h) \right\}d\mu\\
&=k\int_{\mt}E_{k}(-\sigma+h)d\mu = -k \int_{\mt}(\sigma-h)(E_k-h^{1/\alpha})d\mu,
\end{align*}
which is a nonpositive quantity. Moreover, the integral is zero only if $E_k$ is constant on the hypersurface, and this only happens for round spheres, see e.g. \cite{Ro}.
\end{proof}

%\begin{proof}By Lemma \ref{dermixv},
%\begin{align*}
%\partial_t V_{n-1}(\Omega_t)&=n(n-1)(n+1)\int_{\mt}E_2(-E_2+h)d\mu\\
%&=\frac{n(n-1)(n+1)}{A(\mt)}\left\{-A(\mt)\int_{\mt}E_2^2d\mu+\left(\int_{\mt}E_2d\mu\right)^2\right\}
%\end{align*}
%By H\"older inequality,
%$$\int_{\mt}E_2d\mu\leq\left(\int_{\mt}E_2^2d\mu\right)^{\frac{1}{2}}A(\mt)^{\frac{1}{2}}$$
%then we find the result.
%\end{proof}

%This property will allow to control some geometrical quantities associated to the hypersurfaces.
% Using \eqref{af} and Lemma \ref{dermixv} we obtain the following corollary.
%\begin{cor}\label{vl}
%There exist constants $\underline V,$ $\overline V>0$ depending only on $\mo$ and $k,n$ such that, along the flow \eqref{fl}, 
%$$\underline V \leq V_{n-k+1}(\Omega_t) \leq \overline V.$$
%\end{cor}
%\begin{proof}
%From  Lemma \ref{dermixv}  and \eqref{af} with $l=n+1$ and $m=n-1$ it follows
%$$V_{n-k+1}(\Omega_0)\geq V_{n-k+1}(\Omega_t)\geq
%\tilde C \ Vol(\Omega_t)^{\frac{n-k+1}{n+1}}=\tilde C \ Vol(\Omega_0)^{\frac{n-k+1}{n+1}},$$
%for a suitable $\tilde C=\tilde C(n,k)>0$.
%\end{proof}
It is now natural to consider the generalized isoperimetric ratio
$$\mathcal{I}_{n-k+1}(\Omega)=\frac{V_{n-k+1}^{n+1}(\Omega)}{Vol^{n-k+1}(\Omega)}.$$
Then, by Lemma \ref{dermixv}, $\mathcal{I}_{n-k+1}(\Omega_t)$ is decreasing along the flow and, in particular, bounded from above. 
We recall that the inner [resp. outer] radius of $\Omega$ is the radius of the biggest $(n+1)$-dimensional sphere contained in $\Omega$ [resp. the smallest $(n+1)$-dimensional sphere that contains $\Omega$]. We indicate inner and outer radii respectively by  $R_-(\Omega)$ and $R^+(\Omega)$.
We need the following property.
\begin{prop}\label{p1}
For any $n\geq1$, $1 \leq k \leq n$ and $c_1>0$ there exist $c_2=c(c_1,n)$ with the following property. Let $\Omega\subset\mathbb{R}^n$ be a compact, convex set with non empty interior such that $\mathcal{I}_{n-k+1}(\Omega)\leq c_1$. Then $\Omega$ satisfies
$$\frac{R^+(\Omega)}{R_-(\Omega)}\leq c_2.$$
\end{prop}
\begin{proof}
We observe that a bound on $\mathcal{I}_{n-k+1}$ implies a bound on the standard isoperimetric ratio involving the area. In fact, we have
$$
\frac{A(\partial\Omega)^{(n+1)}}{Vol(\Omega)^{n}} = \frac{V^{(n+1)}_n(\Omega)}{V^n_{n+1}(\Omega)} \leq 
\frac{[V_{n-k+1}(\Omega)]^{\frac{n(n+1)}{n-k+1}}}{V^n_{n+1}(\Omega)^{n}} =[ \mathcal{I}_{n-k+1}(\Omega)]^\frac n{n-k+1}.
$$
The assertion then follows from \cite[Lemma 4.4]{HuSi}, see also \cite[Proposition 5.1]{An1b}. 
\end{proof}

 Let us set $R_{-}(t)=R_{-}(\Omega_t)$ and $R^{+}(t)=R^{+}(\Omega_t)$. By Proposition \ref{conv} we know that the solution of \eqref{fl} stays strictly convex along the flow. Then we can use Proposition \ref{p1} to get the following corollary.
\begin{cor}\label{raggi}
There exist constants $R^+,R_->0$ such that along the flow
$$R_-<R_-(t)\leq R^+(t)<R^+$$ 
\end{cor}
\begin{proof}
 By virtue of the boundedness of the isoperimetric ratio, we can use Proposition \ref{p1} to say that $\frac{R^+(t)}{R_-(t)}$ is uniformly bounded by a constant $c_2$ depending only on $n$, $V_{n-k+1}(\Omega_0)$ and $Vol(\Omega_0)$. By comparing with a ball, we find
  $$Vol(\Omega_t)\leq\kappa_{n+1}(R^+(t))^{n+1} \leq\kappa_{n+1}(c_2R_-(t))^{n+1}  \leq c_2^{n+1}Vol(\Omega_t).$$
Since $Vol(\Omega_t)$ is constant, we obtain bounds from both sides on $R_-(t)$ and $R^+(t)$ which are independent on time.
\end{proof}
\subsection*{Boundedness of the velocity}

Thanks to Corollary \ref{raggi} and Proposition \ref{maxtime}, we are now able to control uniformly the velocity of the flow, and obtain curvature bounds which imply the long time existence for the solution.
\begin{prop}\label{vellim}
There exists a positive constant $C_1$, only depending on $\mo$, such that
$$\sigma(x,t)<C_1$$
for every $(x,t)\in\m\times[0,T)$.
\end{prop}
\begin{proof}
The proof uses a technique introduced in \cite{Tso} and widely used in the following literature. We sketch briefly the procedure for the reader's convenience. We introduce the function
$$W(x,t):=\frac{\sigma(x,t)}{u(x,t)-c}$$
where $u(x,t):=(F(x,t)-\qb,\nu(x,t))$ is the support function centered at a point $\qb$ and $c$ is a positive constant. The lower bound on $R_-(t)$ given by Corollary \ref{raggi} ensures that $\qb$ and  $c>0$ can be chosen in such a way that there is a ball centered at $\qb$ of radius at least $2c$ enclosed by $\mt$ for $t$ in a suitable time interval. After computing the evolution equation satisfied by $W$ and applying the maximum principle, we obtain an upper bound for $W$ which also yields a bound for $\sigma$.

In the volume preserving case, the above argument requires some additional technicalities due to the fact that the hypersurfaces $\mt$ are not enclosed in one another, and so we must choose different centers of the enclosed ball in different time intervals. For the details we refer to 
\cite[\S 7]{An1b}, \cite[\S 4]{Mc1}, \cite[\S 3]{BeSi}.
\end{proof}

If $k>1$, the bound on $\sigma$ provided by the above theorem does not imply that the curvature is bounded. In fact, there remains the possibility that some principal curvatures become unbounded while others tend to zero. However, we can already exclude this behaviour on any finite time interval, and obtain that the solution exists for all times. We begin by estimating the mixed volumes.

\begin{cor}\label{hlim} 
All mixed volumes $V_i(\Omega_t)$ are bounded from above and below by positive constants uniformly for $t \in [0,T)$. Similarly,
there are two constants $\beta,\gamma>0$, only depending on $\mo$ such that, on $[0,T)$
$$\beta\leq h(t)\leq \gamma.$$
\end{cor}
\begin{proof}
The bound from below follows  from \eqref{af} and the volume preserving property
$$
V_i(\Omega_t) \geq C \, {\rm Vol}(\Omega_t)^{\frac{n-i}{n+1}} =  C \, {\rm Vol}(\Omega_0)^{\frac{n-i}{n+1}}.
$$
Here we denote by $C$ all constants depending on $i,n$ but not on $t$. Inequalities \eqref{af} also give a bound from above for $n-k+1 \leq i \leq n$, thanks to Lemma \ref{dermixv}. In the case $1 \leq i \leq n-k$, we can use Lemma \ref{dispolsim} and
Proposition \ref{vellim} to obtain
$$
V_i(\Omega_t)= C \int_{\mt} E_{n-i} d\mu \leq C \int_{\mt} E_k^{\frac{n-i}{k}} d\mu \leq C A(\mt) = C V_n(\Omega_t) \leq C.
$$

%The last assertion follows observing that $h(t)= V_{n-k}(\Omega_t)/V_n(\Omega_t)$.
%Since the $\mt$'s are convex and both the inner and outer radii are bounded from above and below, then $A(\mt)$ can not to be too small or too big. Then, Proposition \ref{vellim} implies that $h(t)$ is uniformly bounded form above. Furthermore, by \eqref{af}, 
%$$ 
%\int_{\mt} E_k d\mu = C \, V_{n-k}(\Omega_t) \geq C \, {\rm Vol}(\Omega_t)^{\frac{n-k}{n+1}} =  {\rm Vol}(\Omega_0)^{\frac{n-k}{n+1}},
%$$
%where all constants depending on $k,n$, are denoted by $C$. The proof is complete.
An upper bound for $h(t)$ follows from Proposition \ref{vellim}. Since $A(\mt)=V_n(\mt)$ is bounded from both sides, a bound from below on $h(t)$ is equivalent to a bound on $\int_{\mt}\sigma d\mu$. Let $\eta>0$, and set $\tilde{\mt}=\{x\in\m\hspace{1mm}|\hspace{1mm}E_k(x,t)\geq\eta\}$. Then,
\begin{align*}
C&\leq V_{n-k}(\Omega_t)=C\int_{\mt}E_kd\mu=C\int_{\tilde{\mt}}E_kd\mu+C\int_{\mt\smallsetminus\tilde{\mt}}E_kd\mu\\
&\leq C A(\tilde{\mt})+C\eta A(\mt)\leq C A(\tilde{\mt})+C\eta.
\end{align*}
If we choose $\eta$ suitably small, we deduce 
$$A(\tilde{\mt})\geq C$$
and we can conclude 
$$\int_{\mt}\sigma d\mu\geq \int_{\tilde{\mt}}\sigma d\mu\geq\eta^{\alpha}A(\tilde{\mt})\geq C.$$
\end{proof}

We can now prove that the solution to \eqref{fl} exists for all times.

\begin{theorem}
The solution $\mt$ of the flow \eqref{fl} exists for $t \in [0,+\infty)$.
\end{theorem}
\begin{proof} Suppose that the maximal time $T$ is finite. By Proposition \ref{conv} and Corollary \ref{hlim}, we obtain that the principal curvatures are bounded from below for all $t \in [0,T)$ by some constant $\lambda_0$. It follows, using Proposition \ref{vellim},
$$
\ld_n = \frac{\ld_{n-k+1} \cdots \ld_n}{\ld_{n-k+1} \cdots \ld_{n-1}} \leq \frac{E_k}{\lambda_0^{k-1}} \leq \frac{C_1^{\frac{1}{\alpha}}}{\lambda_0^{k-1}},
$$
which shows that the curvatures are also bounded from above on $[0,T)$. This contradicts Corollary \ref{maxtime}
and shows that $T$ is infinite.
\end{proof}

%\section{Hausdorff convergence}

\begin{theorem}\label{mean}
For a general $\alpha>0$, we have $\liminf_{t \to +\infty} \int_{\mt} |\sigma-h(t)|^2 \, d\mu = 0$. If $\alpha=1$, we  have $\lim_{t \to +\infty} \int_{\mt} |\sigma-h(t)|^2 \, d\mu = 0$.
\end{theorem}
\begin{proof}
From the proof of Lemma \ref{dermixv} we know that
$$\int_0^{\infty}\left(\int_{\mt}|\sigma-h||E_k-h^{1/\alpha}|d\mu\right)dt<+\infty.$$
If $0<\alpha\leq1$, it can be easily checked that
$$|\sigma-h|\leq\frac{|E_k-h^{1/\alpha}|}{h^{(1-\alpha)/\alpha}}\leq\beta^{(\alpha-1)/\alpha}|E_k-h^{1/\alpha}|,$$
where the last inequality comes from Corollary \ref{hlim}.
If $\alpha\geq1$, then from Proposition \ref{vellim} and Corollary \ref{hlim} it follows that
$$|\sigma-h|\leq \alpha (\max \{E_k,h^{1/\alpha}\} )^{\alpha-1}|E_k-h^{1/\alpha}|\leq  C|E_k-h^{1/\alpha}|$$
for some constant $C>0$. We conclude that, for any $\alpha>0$, we have
\begin{equation}\label{impr}
\int_0^{\infty} \left (\int_{\mt} |\sigma-h|^2 d\mu  \right) \, dt\leq
 C'\int_0^{\infty} \left (\int_{\mt} |\sigma-h| |E_k-h^{\frac{1}{\alpha}}|d\mu  \right) \, dt< +\infty,
\end{equation}
for some $C'>0$. This implies
$$
\liminf_{t \to +\infty} \int_{\mt} |\sigma-h|^2 d\mu =0.
$$

In the case $\alpha=1$, we have a more precise result by estimating the time derivative of the integral. We have
$$
\int_{\mt} |\sigma-h(t)|^2 d\mu = \int_{\mt} \sigma^2 d\mu - \frac{1}{|\mt|}  \left( \int_{\mt} \sigma d\mu \right)^2.
$$
For $\alpha=1$ the operator $\ls$ is self-adjoint and we find, by Proposition \ref{eveq} and \ref{dispolsim},
\begin{align*}
\frac{d}{dt}\int_{\mt}\sigma d\mu&=\int_{\mt}(\sigma-h)(\ts(h_{ik}h^k_j)-H\sigma)d\mu\\
&=-(k+1)\int_{\mt}E_{k+1}d\mu.
\end{align*}
Since $E_{k+1}$ is uniformly bounded,  as well as the area of $\mt$, then
$$\left|\frac{d}{dt}\int_{\mt}\sigma d\mu\right|\leq C.$$
%which is also uniformly bounded, since the integral of $H$ is equal to $V_{n-1}(\Omega_t)$ up to a constant factor.
In addition, we have
$$
\frac{d}{dt} |\mt| = - \int_{\mt} H(\sigma-h) d\mu.
$$
Therefore
$$
\left| \frac{d}{dt} |\mt|  \right| \leq  C \int_{\mt} H d\mu,
$$
which is uniformly bounded,  since the integral of $H$ is equal to $V_{n-1}(\Omega_t)$ up to a constant factor. Finally we compute
$$\frac{d}{dt}\int_{\mt}\sigma^2d\mu=
\int_{\mt}\left(-2|\nabla\sigma|^2_{\dot{\sigma}}+\sigma(\sigma-h)\ts(h_{ik}h^k_j)
-\sigma H(\sigma-h)\right)d\mu,$$
where $| \nabla E_k |^2_{\dot \sigma} = \dot \sigma^{ij} \nabla_i E_k \nabla_j E_k$.
The gradient term gives a negative contribution, while all the remaining terms have a bounded integral by similar arguments as before.
It follows that we can find an upper bound
\begin{equation}\label{onesided}
\frac{d}{dt} \int_{\mt} |\sigma-h|^2 d\mu \leq C,
\end{equation}
where $C$ does not depend on $t$.

Let us set $l:=\limsup_{t \to +\infty} \int_{\mt} |\sigma-h|^2 d\mu$. If $l>0$, then \eqref{impr} implies that $\int_{\mt} |\sigma-h|^2 d\mu$ oscillates infinitely many times between $0$ and $l$ with an arbitrarily large speed as $t \to \infty$. However, the one-sided bound \eqref{onesided} is enough to exclude that $\int_{\mt} |\sigma-h|^2 d\mu$ has arbitrarily fast oscillations. Therefore the integral must tend to zero.

\end{proof}

\section{Smooth convergence of the scalar curvature flow}

%\subsection*{Uniform upper bound for the curvature} 
We now restrict to the case $k=2$ and $\alpha=1$, where the speed is given by the scalar curvature. In this case, the boundedness of the speed allows us to prove that all principal curvatures are bounded, as shown in the next theorem.

\begin{theorem} \label{curvlim}
There exists a constant  $C_2>0$ such that on $[0,\infty)$
$$\lambda_i\leq C_2\hspace{5mm}\forall i=1,\dots,n.$$
\end{theorem}
\begin{proof}
We can rewrite the evolution of $H$ as in Corollary $4.2$ of \cite{AlSi} :
\begin{equation}\label{evolmc}
\partial_t H= \ls H+|\nabla H|^2-|\nabla A|^2-E_2|A|^2+\left(H|A|^2-({\rm tr}(A^3) \right)H -h|A|^2
\end{equation}
where ${\rm tr}(A^3)=\sum_{i=1}^n \lambda_i^3$. At a local maximum point for $H$, the terms containing derivatives are non positive. Let us analyse the reaction terms.

Since $E_2\leq C_1$, we can estimate % then $\lambda_i\leq \frac{C_1}{\lambda_n}$ for every $i<n$, and we can write
\begin{align*}
H|A|^2-{\rm tr}(A^3) &=|A|^2\sum_{i=1}^{n-1}\lambda_i+\sum_{i=1}^{n-1}(\lambda_n-\lambda_i)\lambda_i^2
\leq|A|^2\sum_{i=1}^{n-1}\lambda_i+\lambda_n\sum_{i=1}^{n-1}\lambda_i^2\\
&\leq|A|^2\sum_{i=1}^{n-1}\lambda_i+(n-1)\lambda_n\lambda_{n-1}^2
\leq|A|^2\sum_{i=1}^{n-1}\lambda_i+(n-1)C_1\lambda_{n-1}.
\end{align*}
Then we obtain
\begin{align*}
-E_2|A|^2+\left(H|A|^2-({\rm tr}(A^3) \right)H&\leq-\lambda_n|A|^2\sum_{i=1}^{n-1}\lambda_i + H|A|^2\sum_{i=1}^{n-1}\lambda_i \\
& \hspace{4mm} +(n-1)nC_1\lambda_{n-1}\lambda_n\\
&=|A|^2(H-\lambda_n)^2+(n-1)nC_1\lambda_{n-1}\lambda_n\\
& \leq  n \lambda_n^2 (n-1)^2 \lambda_{n-1}^2+(n-1)nC_1\lambda_{n-1}\lambda_n \\
& = (n-1)n\left\{(n-1)(\lambda_n\lambda_{n-1})^2+C_1\lambda_n\lambda_{n-1}\right\}\\
%&\leq (n-1)n\left\{(n-1)E_2^2+C_1E_2\right\}\\
&\leq (n-1)n^2C_1^2.
\end{align*}
We conclude from equation \eqref{evolmc} that, at any local maximum of $H$,
$$\partial_t H \leq(n-1)n^2C_1^2-\frac{\beta}{n}H^2$$
with $\beta$ as in Corollary \ref{hlim}. The maximum principle implies
$$H(x,t)\leq\max\left\{\max_{\m_0}H,nC_1\sqrt{\frac{(n-1)n}{\beta}}\right\}$$
at any time $t\in[0,\infty)$. Since $\mt$ is convex, the same bound holds for any principal curvature.
\end{proof}

To prove convergence to a sphere, we will adapt the strategy used by Ros \cite{Ro} to prove that any closed embedded hypersurface with constant $E_k$ is a sphere. The most delicate step is the result of the next Lemma, since the inverse of $E_2$ may in principle become arbitrarily large as time increases. We have to make a careful use of Proposition \ref{conv}, which gives a control on the rate at which the curvatures can decrease. 

\begin{lem}\label{ros1}
There exists a sequence of times $\{t_l\}$, with $t_l \to +\infty$ such that
$$
\int_{\m_{t_l}} \left| \frac{1}{E_2^{1/2}} - \frac{1}{h^{1/2}} \right| \, d\mu \to 0 \quad \mbox{ as }l \to \infty.
$$
\end{lem}
\begin{proof}
We recall from the proof of Lemma \ref{dermixv} that
$$
\frac d{dt} \int_{\mt} H \, d\mu = -2 \int_{\mt} \left| E_2 - h \right|^2 \, d\mu.
$$
This implies that, for any given integer $l>1$, we cannot have $\int_{\mt} \left| E_2 - h \right|^2 \, d\mu > \frac{1}{lt}$ for all large $t$. In particular, we can find a time $t_l \geq l$ such that
\begin{equation}\label{september}
\int_{\m_{t_l}} \left| E_2 - h \right|^2 \, d\mu \leq \frac{1}{ l t_l}.
\end{equation}
Since the area of $\mt$ is bounded by Corollary \ref{hlim} we have, for some constant $C$, 
\begin{equation}\label{september5}
\int_{\m_{t_l}} \left| E_2 - h \right| \, d\mu \leq A(\mt)^{\frac 12} \left( \int_{\m_{t_l}} \left| E_2 - h \right|^2 \, d\mu \right)^{\frac 12}\leq \frac{C}{ \sqrt{l t_l}}.
\end{equation}
For any fixed $t_l$, we want to estimate the measure of the subset of $\m_{t_l}$ where $E_2$ is small compared with its mean value, say $E_2 < h/2$. We have
\begin{eqnarray*}
\int_{\m_{t_l}} \left| E_2 - h \right|^2 \, d\mu & = & \int_{\{ E_2 \leq h/2 \}}  \left| E_2 - h \right|^2 \, d\mu + \int_{\{ E_2 > h/2 \} } \left| E_2 - h \right|^2 \, d\mu \\
& > & \int_{\{ E_2 \leq h/2 \} } \left| E_2 - h \right|^2 \, d\mu \geq \mu \left( \{ E_2 \leq h/2 \} \right) \frac{h^2}{4}.
\end{eqnarray*}
By \eqref{september} and by Corollary \ref{hlim}, this implies
\begin{equation}\label{september2}
\mu \left( \, \left\{ x \, : \, E_2(x,t_l) \leq \frac{h(t_l)}2  \right\} \,\right) 
\leq
\frac{4}{\beta^2} \frac{1}{ l t_l}.
\end{equation}
We further observe that, by Proposition \ref{conv} and Corollary \ref{hlim}, we have
$$
\min_{\m_{t_l}} E_2 > \min_{\m_{t_l}} \lambda_1 \lambda_2 > \frac{1}{(\lambda_{\min}(0)^{-1}+\gamma t_l)^2}.
$$
Then we conclude
\begin{eqnarray*}
\lefteqn{ \int_{\m_{t_l}} \left| \frac{1}{E_2^{1/2}} - \frac{1}{h^{1/2}} \right| \, d\mu  = 
\int_{\m_{t_l}}  \frac {|E_2 - h|}{ E_2^{1/2}  h^{1/2} ( E_2^{1/2} + h^{1/2})} \, d\mu} \\
& \leq &  \int_{\{ E_2 \leq h/2 \}}  \frac {|E_2 - h|}{ E_2^{1/2} h } \, d\mu 
+ \int_{\{ E_2 > h/2 \}}  \frac {|E_2 - h|}{ E_2^{1/2} h } \, d\mu \\
& \leq & \frac{1}{(\min E_2^{1/2})} \mu \left( \{ E_2 \leq h/2 \} \right) 
+ \frac{\sqrt 2}{h^{3/2}} \int_{\m_{t_l}}  |E_2 - h| \, d\mu \\
& \leq & (\lambda_{\min}(0)^{-1}+\gamma t_l) \frac{4}{\beta^2} \frac{1}{l t_l} + \frac{\sqrt 2}{\beta^{3/2}} \frac{C}{\sqrt{l t_l}}\\
&\leq& \frac{4 (\lambda_{\min}(0)^{-1}+\gamma)}{\beta^2 l} + \frac{\sqrt 2 C}{\beta^{3/2}l} \longrightarrow 0 \mbox{ as }l \to \infty.
%& \leq & \frac{\gamma t_i }{h(t_i)}  \int_{\m_{t_i}} {|E_2 - h|^2 } \, d\mu
\end{eqnarray*}
\end{proof}

\begin{lem}\label{ros2}
There exists a sequence of times $\{t_l\}$, with $t_l \to +\infty$ such that
$$
\int_{\m_{t_l}} \frac{1}{\tilde E_1} \, d\mu - (n+1) Vol(\Omega_{t_l}) \to 0 \quad \mbox{ as }l \to \infty.
$$
\end{lem}
\begin{proof}
Let us set
$$
\tilde h(t) =  \binom{n}{2}^{-1} h(t) = \frac{1}{A(\mt)} \int_{\mt} \tilde E_2 \, d\mu.
$$
We have, using \eqref{minkid} and  the divergence theorem,
\begin{eqnarray*}
\int_{\mt} \tilde E_{1} \, d\mu & = & \int_{\mt} \tilde E_2 (F,\nu) \, d\mu 
\\
& = &  
\tilde h(t) \int_{\mt} (F,\nu) \, d\mu
+  \int_{\mt} \left( \tilde E_2 - \tilde h(t) \right)  (F,\nu) \, d\mu \\
& = &  
(n+1) \, \tilde h(t) \, Vol (\Omega_t)  +  \int_{\mt} \left( \tilde E_2 -\tilde h(t) \right)  (F,\nu) \, d\mu.
\end{eqnarray*}
Up to a translation, we can assume that $\max | (F,\nu) | \leq R^+(t) \leq C$. Therefore, by Theorem \ref{mean}, we have
$$
 \left|  \int_{\mt} \left( \tilde E_2 -\tilde h(t) \right)  (F,\nu) \, d\mu \right| \leq C \int_{\mt} 
 \left| \tilde E_2 -\tilde h(t) \right|  \, d\mu \rightarrow 0 \mbox{ as }t \to \infty.
$$
We deduce
\begin{equation}\label{favk2}
\lim_{t \to \infty}  \int_{\mt} \tilde E_1 d\mu - \tilde h(t) (n+1) Vol(\Omega_t)=0.
\end{equation}
On the other hand, by Lemma \ref{dispolsim}
\begin{equation}\label{september3}
\int_{\mt} \tilde E_1 d\mu \geq \int_{\mt} \tilde E_2^{1/2} d\mu = \tilde h(t)^{1/2} A(\mt) + \int_{\mt} (\tilde E_2^{1/2} - \tilde h(t)^{1/2}) d\mu.
\end{equation}
Since we have $|\sqrt a - \sqrt b| \leq \frac{|a-b|}{\sqrt b}$ for any numbers $a,b>0$, we deduce from Corollary \ref{hlim} and Theorem \ref{mean}
$$
\int_{\mt} |\tilde E_2^{1/2} - \tilde h(t)^{1/2}|d\mu \leq  \frac{1}{\beta^{1/2}}\int_{\mt} |\tilde E_2 - \tilde h(t)| d\mu \rightarrow 0 \mbox{ as }t \to \infty.
$$
Therefore \eqref{september3} implies
$$
\liminf_{t \to \infty} \int_{\mt} \tilde E_1 d\mu - \tilde h(t)^{1/2} A(\mt) \geq 0.
$$
Together with \eqref{favk2}, we find
\begin{equation}\label{ros3}
\liminf_{t \to \infty} \left( (n+1) Vol(\Omega_t) - \frac{A(\mt) }{ \tilde h(t)^{1/2} } \right) \geq 0.
\end{equation}
On the other hand, by Lemma \ref{dispolsim}, we have
$$
\int_{\mt} \frac{1}{\tilde E_1} \, d\mu \leq \int \frac{1}{\tilde E_2^{1/2}} \, d\mu = \frac{A(\mt)}{\tilde h(t)^{1/2}} + \int_{\mt} \left(  \frac{1}{\tilde E_2^{1/2}} - \frac{1}{\tilde h^{1/2}} \right) \, d\mu.
$$
If we pick the sequence $\{t_l\}$ such that Lemma \ref{ros1} holds, we find
$$
\liminf_{l \to \infty}  \frac{A(\m_{t_l})}{\tilde h(t_l)^{1/2}} - \int_{\m_{t_l}} \frac{1}{\tilde E_1} \, d\mu \geq 0.
$$
By virtue of \eqref{ros3}, we conclude
$$
\liminf_{l \to \infty} (n+1) Vol(\Omega_{t_l})- \int_{\m_{t_l}} \frac{1}{\tilde E_1} \, d\mu \geq 0.
$$
On the other hand, by Theorem 1 in \cite{Ro}, any smooth closed hypersurface $\m$ with $\m=\partial \Omega$ satisfies the reverse inequality
$$
\int_{\m} \frac{1}{\tilde E_1} \, d\mu \geq (n+1) Vol(\Omega).
$$
From the two last inequalities, the assertion follows.
\end{proof}

\begin{remark}
{\rm Equality $\int_\m \tilde E_1 ^{-1} \, d\mu = (n+1) Vol(\Omega)$ characterizes the sphere, as shown in \cite{Ro}. It can be checked that the proof of Lemma \ref{ros2} also holds, with some additional computation, in the case of general $k,\alpha$. The next estimate, instead, will make essential use of the curvature bound of Theorem \ref{curvlim}.}
\end{remark}

\begin{prop}
\label{convr}
We have $\lim_{t \to \infty} R^+(t) = \lim_{t \to \infty }R_-(t)=R$, where $R$ is the radius of a sphere enclosing the same volume as $\m_0$.
\end{prop}
\begin{proof}
We first show that $R^+(t)-R_-(t)$ tends to zero along the time sequence given by the previous lemma. This is an easy consequence of a stability estimate recently proved by Magnanini and Poggesi in \cite{MaPo}. We recall here their result, with some simplifications due to the fact that we are dealing with convex sets.

For a fixed $t>0$, we take $v:\overline{\Omega}_t\rightarrow \mathbb{R}$ as the solution of
\begin{equation}
\label{v}
\begin{cases}
\Delta v=n+1  \hspace{3mm}\text{ on } \Omega_t\\
v=0 \hspace{3mm} \text{ on } \mt .
\end{cases}
\end{equation}
We then consider the function $w(x)=q(x)-v(x)$, where $q(x)=\frac{1}{2}(|x-z|^2-a)$ for some fixed $z\in\mathbb{R}^{n+1}$ and $a\in\mathbb{R}$. By Theorem $2.6$ and formula $(3.2)$ in \cite{MaPo}, the hessian matrix of $w$ satisfies the inequality 
\begin{equation}
\label{dis2}
\frac{1}{n}\int_{\Omega_t}|Hess(w)|^2 dV \leq\int_{\mt}\frac{1}{\tilde E_1}d\mu-(n+1)Vol(\Omega_t).
\end{equation}
Then we have the following estimate, which follows from Theorem $3.4$, Lemma $3.7$ and Theorem $3.10$ of \cite{MaPo}: there exist constants $\bar C, \bar \varepsilon>0$ such that, if we choose $z$ in the definition of $w$ as a local minimum point of $v$ in $\Omega_t$, we have
\begin{equation}\label{magpog}
R^+(t)-R_-(t)\leq \bar C ||Hess(w)||^{\frac{2}{n+3}}_{L^2(\Omega_t)}\hspace{.5cm}\text{if}\hspace{.5cm}
||Hess(w )||_{L^2(\Omega_t)}<  \bar \varepsilon.
\end{equation}
Here, the constants $\bar C, \bar \varepsilon>0$ depend on the diameter and the volume of $\Omega_t$ and on the {\em inner uniform radius} of $\mt$, defined as
$$r_i(t)=\inf_{p\in \mt} \sup\{r>0 |\exists q\in\Omega_t \text{ such that }
B_r(q) \text{ touches } \mt \text{ at }p \text{ from inside} \}.$$
For a convex $\Omega_t$, such a quantity is controlled by the inverse of the curvature of $\mt$, and therefore is uniformly bounded from below by Theorem \ref{curvlim}. The diameter and the volume of $\Omega_t$ are also uniformly bounded, by the results of the previous sections. It follows that estimate \eqref{magpog} holds with constants $\bar C, \bar \varepsilon>0$ independent of $t$. In view of estimates \eqref{dis2}, \eqref{magpog}, we obtain from Lemma \ref{ros2} that $R^+(t_l)-R_-(t_l) \to 0$ as $l \to +\infty$.

The convergence of the radii, together with the volume constraint, shows that the sets $\Omega_{t_l}$ converge in the Hausdorff metric to a sphere $B_R$ with the same volume as $\Omega_0$. By the continuity of the mixed volumes with respect to the Hausdorff convergence, the generalized isoperimetric ratio $\mathcal{I}_{n-k+1}(\Omega_{t_l})$ tends to the value of the sphere $\mathcal{I}_{n-k+1}(B)$, which is the smallest possible value by Alexandrov-Fenchel's inequality. Since $\mathcal{I}_{n-k+1}(\mt)$ is monotone decreasing by Lemma \ref{dermixv}, it follows that
\begin{equation}\label{september4}
\lim_{t \to \infty }\mathcal{I}_{n-k+1}(\mt) = \inf_{t \geq 0} \mathcal{I}_{n-k+1}(\mt) = \mathcal{I}_{n-k+1}(B_R).
\end{equation}

Let us now suppose, for the sake of contradiction, that $\limsup_{t \to \infty} R^+(t)-R_-(t)>0$. Then, by our diameter bounds and by Blaschke's compactness theorem for convex sets, we could find a sequence of times $\{t_h\}$ such that $\Omega_{t_h}$, up to translations, converge in the Hausdorff metric to a limit set $\Omega_\infty$ which is not a sphere. By \eqref{september4} and by the continuity of mixed volumes, such a set should have the same generalized isoperimetric ratio as the sphere. This is a contradiction, since the best constant in \eqref{af} is only attained by the spheres.
\end{proof}

The convergence of the radii allows us to use the same argument as in \cite{BeSi} to prove a bound from below on the speed. This will imply the uniform parabolicity of the flow and the regularity required to conclude the proof Theorem \ref{mt}.

\begin{prop}\label{limbv}
There exists a positive constant $C_3$, only depending on $n$ and $\mo$, such that
$$E_2(x,t)>C_3$$
for every $(x,t)\in\m\times[0,\infty)$.
\end{prop}
\begin{proof}
By the previous proposition, for any $\varepsilon>0$, there exists $T_\varepsilon$ such that, for any $t_0 \geq T_\varepsilon$, there exists a point $q=q(t_0)$ such that
$$
B_{R-\varepsilon}(q) \subset \Omega_{t_0} \subset B_{R+\varepsilon}(q).
$$
Since the speed is bounded, there exists $\tau=\tau(\varepsilon)$ such that
$$
B_{R-2\varepsilon}(q) \subset \Omega_{t} \subset B_{R+2\varepsilon}(q), \qquad t \in [t_0,t_0+\tau].
$$
If we now consider the support function $u=(F-q,\nu)$ and we set $c= R+3\varepsilon$, we have
$$
\varepsilon \leq c-u \leq 5 \varepsilon
$$
on $\mt$, for every $t \in [t_0,t_0+\tau]$. On this time interval, we consider the function
$$W(x,t)=\frac{E_2(x,t)}{c-u(x,t)}.$$
Standard computations show that
\begin{align*}
(\partial_t-\ls)W&=\frac{2}{c-u}\langle\nabla u,\nabla W\rangle_{\dot{\sigma}}-3W^2-\frac{cW}{c-u}(HE_2-3E_3)\\
&\hspace{4mm}+\frac{h}{c-u}W-\frac{h}{c-u}(HE_2-3E_3)\\
&\geq\frac{2}{c-u}\langle\nabla u,\nabla W\rangle_{\dot{\sigma}}-W^2(3+cH)+Wh\left(\frac{1}{c-u}-H\right).
\end{align*}
Let $\bar H$ denote the supremum of $H$ along the flow, and let us choose $\varepsilon=(10 \bar H)^{-1}$, so that
$$
\frac{1}{c-u}-H \geq \frac{1}{5 \varepsilon} - \bar H = \bar H.
$$
Then, at any point where the minimum of $W$ on $\mt$ is attained, we have
$$
\partial_t W  \geq-W^2(3+c \bar H)+ W h \bar H \geq W(\beta \bar H - W(4+R \bar H)).
$$
This shows that $W$ cannot attain a new minimum smaller than $\frac{\beta\bar H}{4+R \bar H}$ at a time $t \geq T_\varepsilon$, and implies that $E_2$ is bounded from below by a positive constant for all times.
\end{proof}

From Proposition \ref{limbv}, it follows that at least two principal curvatures are uniformly bounded from below, i.e. there exists $\lambda>0$ such that
 $$\lambda_{n-1}(x,t),\lambda_n(x,t)>\lambda \hspace{0.5cm}\text{ for all } (x,t)\in\m\times[0,\infty).$$
 Then the operator $\Delta_{\dot{\sigma}}$ is uniformly parabolic for all $t \in [0,\infty)$: In fact, given $\omega=(\omega_1\ldots,\omega_n)\in\mathbb{R}^n$,
 $$\dot{\sigma}^{ij}\omega_i\omega_j=\frac{\partial E_2}{\partial\lambda_i}\omega_i^2=(H-\lambda_i)\omega_i^2
 \geq(H-\lambda_n)|\omega|^2\geq\lambda_{n-1}|\omega|^2>\lambda|\omega|^2.$$
 Arguing as in the proof of Theorem $6.4$ in \cite{CaSi} and Proposition $4.3$ in \cite{BeSi}, we find that all the derivatives of the curvatures are bounded on $[0,\infty)$. Therefore, the Hausdorff convergence of the $\mt$'s to a sphere is also a convergence in the $C^\infty$ norm.
 
Finally, in order to obtain the exponential rate of the convergence we can observe that, after a certain time $t^*$, the pinching condition $(1.6)$ appearing in \cite{CaSi} holds. Then we can apply Theorem $7.7$ of that paper to conclude that the hypersurfaces $\mt$ converge exponentially to a round sphere, with no need to add space isometries. 
The proof of Theorem \ref{mt} is complete.

\bigskip

\noindent {\bf Acknowledgments} 
Carlo Sinestrari was partially supported by the research group GNAMPA of INdAM (Istituto Nazionale di Alta Matematica).


\begin{thebibliography}{20}

\bibitem{AlSi}R. Alessandroni, C. Sinestrari, {\em Evolution of hypersurfaces by powers of the scalar curvature}, Ann. Scuola Norm. Sup. Pisa Cl. Sci. \textbf{9} (2010), 541--571.

\bibitem{AM} B. Andrews, J. McCoy, {\em Convex hypersurfaces with pinched principal curvatures and flow of convex hypersurfaces by high powers of curvature}, Trans. Amer. Math. Soc. \textbf{364} (2012), 3427--3447.

\bibitem{An1}B. Andrews, {\em Contraction of convex hypersurfaces in Euclidean space}, Calc. Var. Partial Differ. Equ. \textbf{2} (1994), 151--171.

\bibitem{An1b}B. Andrews, {\em Volume-preserving anisotropic mean curvature flow}, Indiana Univ. Math. J. \textbf{50} (2001), 783--827.

\bibitem{An2} B. Andrews, {\em Gauss curvature flow: the fate of the rolling stones}, Invent. Math. \textbf{138} (1999), 151--161.

%\bibitem{An3} B. Andrews, {\em Pinching estimates and motion of hypersurfaces by curvature functions},  J. Reine Angew. Math. \textbf{608} (2007), 17--33.

%\bibitem{An4 } B. Andrews, Fully nonlinear parabolic equations in two space variables} arXiv:0402235 (2004).

\bibitem{AnMcZh} B. Andrews, J. McCoy, Y. Zheng, {\em Contracting convex hypersurfaces by curvature},  Calc. Var. Partial Differ. Equ. \textbf{47} (2013), 611--665.

\bibitem{AW} B. Andrews, Y. Wei, {\em Volume preserving flows by powers of the $k$-th mean curvature}, arXiv:1708.03982v1 (2017).

\bibitem{BeSi} M.C. Bertini, C. Sinestrari {\em Volume preserving non homogeneous mean curvature flow of convex hypersurfaces}, arXiv:1610.07436 (2016).

%\bibitem{BoFe} T. Bonnesen, W. Fenchel, Theorie der konvexen K\"orper, Springer, Berlin (1934).

\bibitem{BuZa} Yu. D. Burago, V.A. Zalgaller, Geometric inequalities, Springer-Verlag, Berlin Heidelberg (1988).

 \bibitem{Ca} E. Cabezas-Rivas, {\em Volume preserving curvature flows in rotationally symmetric spaces}, PhD thesis, University of Valencia (2008).


\bibitem{CaSi}E. Cabezas-Rivas, C. Sinestrari, {\em Volume-preserving flow by powers of the $m$-th mean curvature}, Calc. Var. Partial Differ. Equ. \textbf{38} (2010), 441--469.

\bibitem{Ch} B. Chow, {\em Deforming convex hypersurfaces by the nth root of the Gaussian curvature}, J. Diff. Geom. \textbf{22} (1985), 117--138.

%\bibitem{Ha} R.S.Hamilton, {\em Three-manifolds with positive Ricci curvature}, J. Differ. Geom. \textbf{17} (1982), 255--306.
 
\bibitem{Hs} C.C.Hsiung, {\em Some integral formulas for closed hypersurfaces}, Math. Scand. \textbf{2} (1954), 286--294.

\bibitem{Hu1}G.Huisken, {\em Flow by mean curvature of convex surfaces into spheres}, J. Differ. Geom. \textbf{20} (1984), 237--266.

\bibitem{Hu2} G.Huisken, {\em The volume preserving mean curvature flow}, J. Reine Angew. Math. \textbf{382} (1987), 35--48.

\bibitem{HuPo}G.Huisken, A.Polden, {\em Geometric evolution equations for Hypersurfaces}, Calculus of Variations and Geometric Evolutions Problems (Cetraro, 1996), Lecture Notes in Mathematics, vol. 1713 (1999), pp 45--84 Springer, Berlin.

%\bibitem{HuSi2}G.Huisken, C.Sinestrari, {\em Convexity estimates for mean curvature flow and singularities of mean convex surfaces}, Acta Math. \textbf{183}  (1999), 45--70.

\bibitem{HuSi}G.Huisken, C.Sinestrari, {\em Convex ancient solutions of the mean curvature flow}, J. Differ. Geom. \textbf{101} (2015), 267--287.

%\bibitem{Ga}L.G\"arding, {\em Al inequality for hyperbolic polynomials}, J.Math.Mech. \textbf{8} (1959), 957--965.

\bibitem{Li}G. M. Lieberman, {\em Second Order Parabolic Differential Equations}, World Scientific (1996).

\bibitem{MaPo}R. Magnanini, G. Poggesi, {\em On the stability for Alexandrov's soap bubble theorem}, arXiv:1610.07036v1 (2016).

\bibitem{Mc1}J.A. McCoy, {\em The mixed volume preserving mean curvature flow}, Math. Z. \textbf{246} (2004), 155--166.

\bibitem{Mc2}J.A. McCoy, {\em Mixed volume preserving curvature flows}, Calc. Var. Partial Differ. Equ. \textbf{24} (2005), 131--154.

\bibitem{Mc3}J.A. McCoy, {\em More mixed volume preserving curvature flows}, J. Geom. Anal. \textbf{27} (2017), 3140--3165.

\bibitem{P} D.M. Pihan, {\em A length preserving geometric heat flow for curves}. PhD thesis, University of Melbourne (1998).

\bibitem{Ro}A. Ros, {\em Compact Hypersurfaces with constant higher order mean curvatures}, Revista Matematica Iberoamericana \textbf{3} (1987), 447--453.

\bibitem{Sc} R. Schneider, Convex bodies: The Brunn Minkowski theory, Encyclopedia of Mathematics and its Applications, Vol. 44. Cambridge University Press, (1993).

\bibitem{Schn} O.C. Schn\"urer, {\em Surfaces contracting with speed $|A|^2$}, J. Differ. Geom. \textbf{71}, (2005), 347--363.

%\bibitem{Sch}F.Schulze, {\em Evolution of convex hypersurfaces by powers of the mean curvature}, Math. Z. \textbf{251} (2005), 721--733.

\bibitem{Sch2}F. Schulze,  {\em Convexity estimates for flows by powers of  the mean curvature} (with an appendix by O. Schn\"urer and F. Schulze), Ann. Sc. Norm. Super. Pisa Cl. Sci. (5) \textbf{5} (2006), 261--277.

\bibitem{Si}C. Sinestrari, {\em Convex hypersurfaces evolving by volume preserving curvature flows}, Calc. Var. Partial Differ. Equ. \textbf{54} (2015), 1985--1993.

\bibitem{Tso}  K. Tso, {\em Deforming a hypersurface by its Gauss-Kronecker curvature},  Comm. Pure Appl. Math. {\bf 38}, (1985) 867--882. 

\end{thebibliography}
\end{document}